\documentclass[a4paper,11pt,reqno]{article}

\usepackage[utf8]{inputenc}
\usepackage[T1]{fontenc}
\usepackage{lmodern}
\usepackage[english]{babel}
\usepackage{microtype}
\usepackage{pdfsync}

\usepackage{amsmath,amssymb,amsfonts,amsthm}
\usepackage{mathtools,accents}
\usepackage{mathrsfs}
\usepackage{aliascnt}
\usepackage{braket}
\usepackage{bm}

\usepackage[a4paper,margin=3cm]{geometry}
\usepackage[citecolor=blue,colorlinks]{hyperref}

\usepackage{enumerate}
\usepackage{xcolor}
\usepackage{xspace}

\makeatletter
\def\newaliasedtheorem#1[#2]#3{
  \newaliascnt{#1@alt}{#2}
  \newtheorem{#1}[#1@alt]{#3}
  \expandafter\newcommand\csname #1@altname\endcsname{#3}
}
\makeatother

\numberwithin{equation}{section}

\newtheoremstyle{slanted}{\topsep}{\topsep}{\slshape}{}{\bfseries}{.}{.5em}{}

\theoremstyle{plain}
\newtheorem{theorem}{Theorem}[section]
\newaliasedtheorem{proposition}[theorem]{Proposition}
\newaliasedtheorem{lemma}[theorem]{Lemma}
\newaliasedtheorem{corollary}[theorem]{Corollary}
\newaliasedtheorem{counterexample}[theorem]{Counterexample}

\theoremstyle{definition}
\newaliasedtheorem{definition}[theorem]{Definition}
\newaliasedtheorem{question}[theorem]{Question}
\newaliasedtheorem{example}[theorem]{Example}

\theoremstyle{remark}
\newaliasedtheorem{remark}[theorem]{Remark}

\newcommand{\setR}{\mathbb{R}}

\newcommand{\Leb}{\mathscr{L}}

\newcommand{\Prob}{\mathscr{P}}

\newcommand{\norm}[2][ ]{\left\| #2 \right\|_{#1}} 
\newcommand{\scal}[2]{\ensuremath{\langle #1 , #2 \rangle}} 

\DeclareMathOperator{\RCD}{RCD}
\DeclareMathOperator{\CD}{CD}
\newcommand{\Ch}{\operatorname{Ch}}

\begin{document}

\title{Lusin-type approximation of Sobolev by Lipschitz functions, in Gaussian and $\RCD(K,\infty)$ spaces}
\author{Luigi Ambrosio
\thanks{Scuola Normale Superiore, \url{luigi.ambrosio@sns.it}.} \and
Elia Bru\`e
\thanks{Scuola Normale Superiore, \url{elia.brue@sns.it}.} \and
Dario Trevisan
\thanks{Università degli studi di Pisa, \url{dario.trevisan@unipi.it}. Partially supported by GNAMPA project 2017 ``Campi vettoriali, superfici e perimetri in geometrie singolari''.}
}
\maketitle

\abstract{We establish new approximation results, in the sense of Lusin, of Sobolev functions by Lipschitz ones, in some classes of non-doubling metric measure structures. Our proof technique relies upon estimates for heat semigroups and applies to Gaussian and $\RCD(K, \infty)$ spaces. As a consequence, we obtain quantitative stability for regular Lagrangian flows in Gaussian settings.}

\tableofcontents
\newpage

\section{Introduction}

We say that a function $f:H\to\setR$ on a metric measure space $(H,d,m)$ is approximable in Lusin's sense by Lipschitz functions if, 
given any $\epsilon>0$,
there exist a Lipschitz function $g:H\to\setR$ and a Borel set  $A\subset H$ such that $m(H\setminus A)<\epsilon$ and $f\equiv g$ on $A$.
In Euclidean metric measure structures it is well known that this property is equivalent to an almost everywhere differentiability, 
in an approximate sense \cite[Thm.~3.1.8]{Federer}. 
A quantitative version of this Lusin-Lipschitz property, namely
\begin{equation}\label{eq:1}
|f(x)-f(y)|\leq d(x,y)\bigl( g(x)+g(y)\bigr)\qquad\text{for some nonnegative $g\in L^p(H,m)$}
\end{equation}
holds for $W^{1,p}$ functions, $p\in (1,\infty)$, in Euclidean spaces (see \cite{liu}). This property, and its adaptation to
$W^{1,1}$ or even $BV$ functions, had in recent years a remarkable range
of applications: lower semicontinuity of integral functionals depending on vector-valued maps \cite{acefu}, rectifiability results in the
theory of currents \cite{AmbrosioKirchheim},
quantitative stability results for flows associated to Sobolev vector fields \cite{crippade}, optimal bounds in the matching problem
\cite{ast}. It is also well-known that in the class of metric measure spaces $(H,d,m)$ satisfying the doubling and $1$-Poincar\'e
inequality, the property \eqref{eq:1}
\textit{characterizes} $W^{1,p}$ functions, while for general metric measure structures 
it is the basis of the definition of the so-called Hajlasz Sobolev functions
(see e.g. \cite{Heinonen}). 

The aim of this paper is to extend this Lipschitz approximation result to some classes of non-doubling metric measure structures. All proofs available so far rely on the doubling condition, and this precludes the possibility to prove the Lusin-Lipschitz
property in Gaussian spaces, not doubling even when they are finite-dimensional (see however \cite{Shapo} for an infinite dimensional result on approximation of vector fields by gradients of Lipschitz functions). Our result, instead,
covers Sobolev functions in Gaussian spaces according to Da Prato \cite{daP}, the Sobolev functions of Malliavin
calculus in Wiener spaces,
and the Sobolev functions in the class of $\RCD(K,\infty)$ metric measure structures introduced in \cite{AGS14}, now object of many
investigations.

In Euclidean spaces the proof of \eqref{eq:1} can be achieved 
writing $f$ as a singular integral
\begin{equation}\label{eq:2}
f(x)=-\int \langle \nabla f(y),\nabla_x G(x,y)\rangle\,dy
\end{equation}
with $G$ fundamental solution of Laplace's operator $\Delta$. In metric measure structures lacking the smoothness
necessary to write \eqref{eq:2}, the strategy is to compare $f$ with a regularization $f_r$: for instance 
$f_r(x)$ could be the mean value of $f$ in the ball $B_r(x)$. Choosing
$r\sim d(x,y)$, $f_r(x)$ is comparable to $f_r(y)$ and the problem reduces to the pointwise estimate of $f(x)-f_r(x)$.
This estimate involves Hardy-Littlewood's maximal function $M(|\nabla f|)(x)$, so that $g\sim M(|\nabla f|)$ and $L^p$ integrability
of $g$ immediately follows by the maximal theorem.  

Clearly, these strategies seem to fail when $m$ is not doubling.
In a (potentially or actually) infinite-dimensional setting, our method is a combination of the two, but uses the semigroup
$R_t$ associated to the Sobolev class $W^{1,2}$ instead of the inversion of Laplace's operator:
our regularization is $f_t=R_t f$, and $t$ will be chosen equal
to $d^2(x,y)$. It follows that we need to estimate
\begin{equation}\label{eq:amb1}
|f(x)-f(y)|\leq |f(x)-R_tf(x)|+|R_t f(x)-R_t f(y)|+|R_tf(y)-f(y)|.
\end{equation}
Roughly speaking the estimates of all terms involve $|\nabla f|$, but while the estimate of the oscillation $|R_t(x)-R_t(y)|$ involves
mostly the curvature properties of the metric measure space, the estimate of $f-R_t f$ is more related to the regularity
of the transition probabilities $p_t(x,y)$ of $R_t$. To illustrate this, we may look at this computation, where $L$ and $\int_H\Gamma(f,g) dm$
are respectively the
infinitesimal generator of $R_t$ and the Dirichlet form associated to $R_t$:  
\begin{eqnarray*}
R_tf(x)-f(x)&=&\int_0^tLR_s f(x) ds=\int_0^t \int_H p_s(x,y) Lf(y) dm(y)ds\\
&=&-\int_0^t\int_H\Gamma(f,p_s(x,\cdot)) dm ds=-2\int_0^t\int_H\Gamma(f,\sqrt{p_s}(x,\cdot))\sqrt{p_s}(x,\cdot) dm ds\\&\leq&
2\int_0^t\bigl(R_s\Gamma(f,f)\bigr)^{1/2}\bigl(\int_H\Gamma(\sqrt{p_s}(x,\cdot),\sqrt{p_s}(x,\cdot))dm\bigr)^{1/2} ds.
\end{eqnarray*}
Even if we knew that $\bigl(\int_H\Gamma(\sqrt{p_s}(x,\cdot),\sqrt{p_s}(x,\cdot))dm\bigr)^{1/2}$ is integrable in $(0,t)$, this would give an estimate with $g\sim \sup_{s>0}\sqrt{R_s\Gamma(f,f)}$ which
would not be enough to deal with $W^{1,2}$ functions, because $\Gamma(f,f)\in L^1(H,m)$ and
 weak-$L^1$ estimates on the operator $\sup_{s>0} |R_s|$
are not available in general. We modify this approach using the formula (see \eqref{fractional estimate})
$$
R_tv-v=\int_0^\infty K(s,t) R_s \sqrt{-L} v \ d s \quad \quad  \forall v\in D(\sqrt{-L}), \quad \forall t\geq 0
$$
for a suitable kernel $K$. This formula
provides the correct integrability estimates, at the price of working with the nonlocal operator $\sqrt{-L}$.

In \autoref{th: Lipschitz approximation scalar} we state how \eqref{eq:1} reads for the three structures of our interest. Since the regularizing properties
of the semigroup are slightly different, the proofs and the statement slightly differ in the three cases. Nevertheless, since
in all the cases the transition probabilities $p(x,\cdot)\in L^1(H,m)$ are naturally defined for \textit{all} $x\in H$, 
we use the induced \textit{pointwise} defined version $R_tf(x)=\int_Hf(y)p_t(x,y) dm(y)$ having, as we illustrate, extra
regularity properties.

Eventually, having in mind the application of the estimate \eqref{eq:1} to vector-valued maps, we provide also
the vector valued counterpart of \autoref{th: Lipschitz approximation scalar}, when
the target is an Hilbert space $E$.

In the final part of the paper we apply the vector-valued version of \eqref{eq:1} to provide an extension to the Gaussian and
Wiener settings of \textit{quantitative} stability results for flows associated to $W^{1,p}$ vector fields $b$. We recover the uniqueness results for flows in Wiener spaces \cite[Thm.~3.1]{ambrosio-fig}, with the exception of the case $p=1$ (and of the $BV$ case in \cite{trevisan-bv}). In Da Prato's setting, we obtain a result that quantitatively improves \cite[Thm.~2.3]{daprato-flandoli-rockner} where they consider the problem of uniqueness of (probability-valued) solutions to the continuity equation. This can be reduced to that of generalized flows by means of a suitable lift using a ``superposition principle'' such as in \cite[Thm.~8.2.1]{ags} (see also \cite[Sec.~3]{trevisan-stepanov} for a similar result in more general settings). One could also obtain quantitative bounds in terms of suitable transportation distances, following e.g.\ the approach in \cite{luo} (in a stochastic setting). We leave to future research possible applications in the theory of flows in $\RCD$ spaces, for which well-posedness and stability, not in quantitative form, are established respectively in \cite{ambrosio-trevisan} and \cite{ambrosio-stra}.

\smallskip 
\noindent
{\bf Acknowledgement.} The first and third authors are grateful to V.\ Bogachev, G.\ Da Prato and M.\ Ledoux  for useful information on this subject.

\section{Notation and preliminary results}

\subsection{Abstract semigroup tools}

Throughout this subsection $G$ denotes a separable Hilbert space.

\begin{proposition}\label{prop: fractional estimate}
Let $R_t=e^{tL}$ be a continuous semigroup acting on $G$, with infinitesimal generator
$L: D(L)\subset G\rightarrow G$.  If $-L$ is a positive selfadjoint operator, one has the representation formula
\begin{equation}\label{fractional estimate}
R_tv-v=\int_0^\infty K(s,t) R_s \sqrt{-L} v \ d s \quad \quad  \forall v\in D(\sqrt{-L}), \quad \forall t\geq 0,
\end{equation}
(understanding the integral in Bochner's sense) for a suitable kernel
$K:\setR_+\times \setR_+\rightarrow \setR$ independent of $R_t$ and satisfying
\begin{equation}\label{eq:goodkernel}
\int_0^\infty |K(s,t)| d s= \frac{4}{\sqrt{\pi}} \sqrt{t} \quad\quad \forall t\geq 0.
\end{equation}
\end{proposition}


\begin{proof} We claim that
\begin{equation}\label{elementary identity}
e^{-bt}-1=\int_0^\infty K(s,t) \sqrt{b}e^{-bs} d s, \quad \quad \forall t\geq 0,\quad \forall b\geq 0
\end{equation}
for some kernel $K$ satisfying \eqref{eq:goodkernel}. For $b>0$ (if $b=0$ \eqref{elementary identity} is obvious), we use the identity
\begin{equation}\label{a1}
e^{-bt}=\frac{1}{\sqrt{\pi}}\int_t^{\infty} \frac{1}{(s-t)^{1/2}}\sqrt{b}e^{-bs} d s, \quad  \quad \forall t\in \setR,
\end{equation}
that follows immediately by
\begin{align*}
\int_t^{\infty} \frac{e^{-bs}}{(s-t)^{1/2}} d s =& e^{-bt}\int_t^{\infty} \frac{e^{-b(s-t)}}{(s-t)^{1/2}} d s\\
=&\frac{ e^{-bt}}{\sqrt{b}}\int_0^{\infty} \frac{e^{-s}}{\sqrt{s}} d s\\
=&\frac{e^{-bt}}{\sqrt{b}}\sqrt{\pi},
\end{align*}
where we have used the well known identity $\int_0^\infty \frac{e^{-s}}{\sqrt{s}} d s=\sqrt{\pi}$.
Using \eqref{a1} we find
$$
e^{-bt}-1=e^{-bt}-e^{-b0}=
\int_{\setR}\frac{1}{\sqrt{\pi}}\left( \frac{\chi_{s>t}}{(s-t)^{1/2}}-\frac{\chi_{s>0}}{s^{1/2}}\right) \sqrt{b}e^{-bs} d s,
$$
so that, setting  
$$
K(s,t):=\frac{1}{\sqrt{\pi}}\left( \frac{\chi_{s>t}}{(s-t)^{1/2}}-\frac{\chi_{s>0}}{s^{1/2}}\right) 
$$
we obtain \eqref{elementary identity}.

Now, to complete the proof of the claim, we have to check that $\int_0^\infty |K(s,t)| d s
=\frac{4}{\sqrt{\pi}} \sqrt{t}$ for every $t\geq 0$ (the case $t=0$ is obvious). Indeed, for $t>0$ we have
\begin{align*}
\int_0^\infty |K(s,t)| d s =& \frac{1}{\sqrt{\pi}}\int_0^t \frac{1}{s^{1/2}} d s+\frac{1}{\sqrt{\pi}} \int_t^{\infty} \frac{1}{(s-t)^{1/2}}-\frac{1}{s^{1/2}} d s\\
=& \frac{2}{\sqrt{\pi}}\sqrt{t}+\frac{2}{\sqrt{\pi}}( (s-t)^{1/2}-s^{1/2})|^{s=\infty}_{s=t}\\
=& \frac{4}{\sqrt{\pi}} \sqrt{t}.
\end{align*}
Using standard notions of functional calculus we can write
$$
R_t=\int_0^\infty e^{-\lambda t} d E(\lambda)\qquad\forall t>0,
$$
where $E$ is the spectral measure associated to $L$. For $v\in D(\sqrt{-L})$, from \eqref{elementary identity} we obtain
\begin{align*}
R_tv-v=& \int_0^\infty( e^{-\lambda t}-1) d  E(\lambda)v\\
=& \int_0^\infty \int_0^\infty K(s,t) \sqrt{\lambda}e^{-\lambda s} d s d E(\lambda)v\\
=& \int_0^\infty K(s,t) \int_0^{\infty}\sqrt{\lambda}e^{-\lambda s}  d E(\lambda)v d s,
\end{align*}
where all integrals are well defined since $v\in D(\sqrt{-L})$ implies $\int_0^\infty\sqrt{\lambda} d E(\lambda)v<\infty$.
We finally observe that 
$$
\int_0^\infty\sqrt{\lambda}e^{-\lambda s}  d E(\lambda)v=R_s \sqrt{-L}v=\sqrt{-L}R_s v \quad\quad\quad \forall v\in D(\sqrt{-L}),
$$
that concludes the proof.
\end{proof}

We now particularize the previous result to the case of Markov symmetric semigroups (see e.g. \cite[pg.~65]{Stein}, 
\cite{bgl}).
Let $(X,\mathcal F,m)$ be an abstract measure space, with $m$ $\sigma$-finite, and let $R_t$ be a symmetric Markov semigroup acting on 
$G=L^2(X,\mathcal F,m)$. In this class of semigroups,  which have a canonical extension to a contraction semigroup in all
$L^p(X,\mathcal F,m)$ spaces, $1<p<\infty$, one can always find, for all $f\in G$, versions of $R_tf$, $t>0$, with the
property that $t\mapsto R_t f(x)$ is continuous (in fact, analytic) in $(0,\infty)$ for $m$-a.e. $x\in X$ (see \cite[pg.~72]{Stein} or
\cite[Thm~8.4.2]{Bogaconv} for a proof). For such continuous version, besides the Littlewood-Paley inequality \cite[pg.~74]{Stein}
\begin{equation}\label{eq:LP}
\int_X \int_0^\infty t\bigl|\frac{d}{dt}R_tf(x)\bigr|^2 dt\,d m(x)\leq \frac 14 \int_X |f|^2 d m,
\end{equation}
we shall also use the following powerful result from the theory of Markov semigroups
(see for instance \cite[pg.~73]{Stein}, or {\color{blue}\cite[Lem.~1.6.2]{bgl}.} 

\begin{theorem}[Maximal inequality]\label{thm:maximal}
For $p\in (1,\infty]$ one has, for some $C_p < \infty$,
$$
\|\sup_{t>0}R_t f\|_p\leq C_p\|f\|_p\qquad\forall f\in L^p(X,\mathcal F,m).
$$
In addition, for all $f\in L^p(X,\mathcal F,m)$, one has $R_tf\to f$
$m$-a.e. as $t\to 0^+$. 
\end{theorem}

For $p \in (1, \infty)$ and $f\in L^p(X,\mathcal F,m)$, we write $f \in D_p(\sqrt{-L})$ if there exists a sequence 
$(f_n) \subset D(\sqrt{-L})$ converging to $f$ in $L^p(X, \mathcal F, m)$ with $(\sqrt{-L} f_n)$ converging to 
some function $g$ in $L^p(X, \mathcal F, m)$. A simple limit argument gives \[ \int_X f \sqrt{-L} h dm=\int_X g h dm,\] 
for all $h\in D(\sqrt{-L})\cap L^{p'}(X,\mathcal F,m)$ with $\sqrt{-L}h\in L^{p'}(X,\mathcal F,m)$.
Therefore, if this class of functions $h$ is dense in $L^{p'}(X,\mathcal F,m)$, $g$ is uniquely determined 
and we can write $g:=\sqrt{-L} f \in L^p(X, \mathcal F, m)$. In our cases of interest the density will be guaranteed
by the validity of the Riesz inequalities in the class $D_2(\sqrt{-L})$ (and then, the definition of $D_p(\sqrt{-L})$ grants
immediately their validity in the class $D_p(\sqrt{-L})$).

\begin{proposition}\label{prop: fractional estimate Markov semigroups}
For every $p \in (1, \infty)$, $f\in D_p( \sqrt{-L})$, one has $m$-a.e. continuous version of
the semigroup $R_t$ satisfying
\begin{equation}\label{fractional inequality on functions}
|R_tf(x)-f(x)|\leq  \frac{4\sqrt{t}}{\sqrt{\pi}} \sup_{s>0} |R_s\sqrt{-L}f|(x) \quad\quad \text{for $m$-a.e. $x\in X$}.
\end{equation}
\end{proposition}
\begin{proof} Assume first $p=2$, i.e.,  $f \in D(\sqrt{-L})$. By \autoref{prop: fractional estimate} we have
\begin{equation}\label{fractional identity on functions}
R_tf(x)-f(x)=\int_0^{\infty} K(r,t) R_s\sqrt{-L}f (x) d r, \quad\quad \text{for $m$-a.e. $x\in X$,}
\end{equation}
so that 
$$
|R_tf(x)-f(x)|\leq \int_{0}^\infty |K(r,t)| d r \sup_{s >0} |R_s \sqrt{-L} f|(x) =  \frac{4\sqrt{t}}{\sqrt{\pi}} \sup_{s>0} |R_s\sqrt{-L}f|(x). 
$$
For a general $p \in (1, \infty)$, $f \in D_p(\sqrt{-L})$, choose a sequence $f_n \in D(\sqrt{-L})$ with $f_n \to f$, $\sqrt{-L}f_n \to \sqrt{-L f}$ in $L^p(X, \mathcal{F}, m)$. Then, the maximal inequality implies that, for $m$-a.e.\ $x \in X$, both $(R_tf_n(x))$ and $(R_t \sqrt{-L}f_n(x))$ converge uniformly with respect to $t \ge 0$. Therefore, both limits provide $m$-a.e.\ continuous representatives of the semigroup and inequality \eqref{fractional inequality on functions} holds.
\end{proof}

\subsection{Da Prato's Sobolev spaces}

A standard reference on this topic is \cite{daP}.
In this setting $X=H$ with $H$ separable Hilbert space endowed with the scalar product $\scal{\cdot}{\cdot}$ and 
$d(x,y)=|x-y|$, where $|\cdot|$ is the norm of $H$. If $m\in\Prob(H)$ is a centered and nondegenerate
Gaussian measure, we denote by $Q\in\mathscr{L}(H;H)$ the covariance operator associated to $m$; by the exponential
integrability of Gaussian measures,  $Q$ is a nonnegative symmetric operator with finite trace. 
For every vector $a\in H$ and for every $Q$ as above, we denote by $N_{a,Q}$ the unique 
Gaussian measure in $H$ with mean $a$ and covariance $Q$ (in particular we often denote $m$ by $N_Q$). 
Let $\mathcal{H}:=Q^{1/2}H$ be the Cameron-Martin space associated to $m$, endowed with the scalar product $(x,\ y)_{\mathcal{H}}:=\scal{Q^{-1/2}x}{Q^{-1/2}y}$ and the induced norm $|x|_{\mathcal{H}}:=|Q^{-1/2} x|$.
We recall the Cameron-Martin formula (see for instance \cite[Thm.~2.8]{daP})
\begin{equation}\label{eq:CM_Wiener}
\frac{dN_{v,Q}}{dN_Q}(x)=\exp\left\lbrace -\frac{1}{2}|v|^2_{\mathcal{H}}+W_{Q^{-1/2}v}(x)\right\rbrace
\quad \quad \forall v\in \mathcal{H},
\end{equation}
where $W$ is the white noise map, that could be defined starting from the linear operator
$$
\tilde{W}: \mathcal{H}\subset H\to L^2(H,m),\qquad \tilde{W}_h(\cdot):=\scal{Q^{-1/2}h}{\cdot},
$$
and using $\|\tilde{W}_h\|_2=|h|$ to extend it to the whole of $H$. Such an extension satisfies $\|\tilde{W}_h\|_2=|h|$ for every 
$h\in H$ and is linear w.r.t. $h$. Moreover the white noise map is exponentially integrable (see \cite[Prop.~1.30]{daP}), precisely
\begin{equation}\label{exponential integrability W}
\int_H e^{W_h(x)} dm(x) =e^{\frac{1}{2}|h|}\qquad\quad \forall h\in H.
\end{equation}
For $1\leq p<\infty$, we now consider the Sobolev space $W^{1,p}(H,m)$ obtained as the closure of smooth cylindrical
 functions with respect to the norm
$$
\norm[W^{1,p}]{u}:=\norm[L^p]{u}+\norm[L^p]{\nabla u}.
$$

In this context the natural semigroup is given by Mehler's formula
\begin{equation}\label{eq:defP_t}
P_tf(x):=\int_H f(e^{At}x+y) dN_{Q_t}(y)=\int_H f(y) d N_{e^{At}x,Q_t}(y),
\end{equation}
where we have set $A:=-\frac{1}{2}Q^{-1}$ (that is an unbounded operator) and 
$$
Q_t:=\int_0^te^{2As} d s=Q(1-e^{2At}).
$$ 
The semigroup $P_t$ is the $L^2(H,m)$ gradient flow associated to the energy 
$\frac{1}{2}\int_H |\nabla u|^2 d m$.

 We shall also use a particular case of
 Cameron-Martin formula
\begin{equation}\label{eq:Cameron_Martin}
\frac{dN_{e^{At}x,Q_t}}{dN_{Q_t}}=
\exp\left\lbrace -\frac{1}{2}|\Gamma_tx|^2+W^t_{\Gamma_tx}(\cdot)\right\rbrace:=\rho_t(x,\cdot)
\quad\text{$N_{Q_t}$-a.e.,}
\end{equation}
for all $x\in H$, where $W^t$ is the white noise map in the Gaussian space $(H, N_{Q_t})$ and
$$
\Gamma_t:=Q^{-1/2}(1-e^{2At})^{-1/2}e^{At}.
$$

Since we aim at pointwise statements, it is important to look, whenever this is possible, for a precise version
of the semigroup. For $P_t$ this is not a problem, since one can use directly \eqref{eq:defP_t} to specify
$P_tf$ pointwise; in addition, since $N_{e^{At}x,Q_t}\ll m$ for all $x\in X$, one has
$P_tf(x)=P_tg(x)$ whenever $t>0$ and $f=g$ $m$-a.e. in $H$. Moreover we have a simple explicit formula for the density $p_t(x,\cdot)$
of  $N_{e^{At}x,Q_t}$ with respect to $m$:
\begin{equation}\label{Da Prato kernel}
p_t(x,\cdot):=\frac{dN_{e^{At}x,Q_t}}{dm}=\frac{dN_{e^{At}x,Q_t}}{dN_{Q_t}}\frac{1}{\sqrt{\text{det}(1-e^{2At})}}\exp\left\lbrace- \frac{1}{2}|\Gamma_t y|^2\right\rbrace,
\end{equation}
see \cite[Lemma.~10.3.3]{daPZ} for a proof.

\begin{theorem} \label{thm:verygood} For any $p\in (1,\infty]$ and $f\in L^p(H,m)$, one has
that $P_tf\in C^\infty(H)$ for every $t\in (0,\infty)$. Moreover the map $t\to P_tf(x)$ is continuous in $(0,\infty)$ for every $x\in H$.
\end{theorem}
For a proof of $P_tf\in C^{\infty}(H)$ we refer to \cite[Thm.~10.3.5]{daPZ} (see also \cite[Thm~8.16]{daP} for the case $p=\infty$).
The continuity of $t\to P_tf(x)$ in $(0,\infty)$, for every $x\in H$, can be easily checked using the formula \eqref{Da Prato kernel} and the identity
$$
P_tf(x)=\int_H f(y) p_t(x,y) dm(y).
$$

A simple consequence of \autoref{thm:verygood} is that 
$P_{s+t}f=P_s (P_tf)$ for all $f\in L^p(H,m)$ in the pointwise sense.

For $f\in W^{1,p}(H,m)$, $1<p\leq\infty$, we shall also use the contractivity property
\begin{equation}\label{eq:contra_Pt}
|\nabla P_t f(x)|\leq e^{-t/(2\lambda_Q)} P_t|\nabla f(x)|\leq P_t|\nabla f|(x)\qquad\forall x\in X,
\end{equation}
where $\lambda_Q$ is the largest eigenvalue of $Q$, and the Riesz inequalities (see
\cite[Thm.~5.2]{Chojinowska-Michalik-Goldys})
\begin{equation}\label{Riesz1}
\norm[L^p]{\nabla f}\leq c_p \norm[L^p]{\sqrt{I-L}f},
\end{equation}
\begin{equation}\label{Riesz2}
\norm[L^p]{\sqrt{I-L}f}\leq c_p\norm[W^{1,p}]{f},
\end{equation}
where $L$ is the infinitesimal generator of $P_t$.

\subsection{Sobolev functions on the Wiener space}

If $H$, $Q$ and $m$ are defined as in the previous subsection, 
in this context, the definition of Sobolev
space $W_{\mathcal{H}}^{1,p}(H,m)$, $1<p<\infty$, takes into account only the derivative along Cameron-Martin 
directions and weights it differently, compared to $W^{1,p}(H,m)$: for every smooth and cylindrical $f:H\rightarrow \setR$ 
and $x\in H$ we consider the linear operator
$$
D_{\mathcal{H}}f(x):\mathcal{H}\rightarrow \setR\quad \quad \quad v\mapsto \frac{\partial f}{\partial v}(x);
$$
we can identify $D_{\mathcal{H}}f(x)$ with a vector in $\mathcal{H}$ (that we still denote by $D_{\mathcal{H}}f(x)$)  by means of the scalar product $(\cdot\ ,\cdot)_{\mathcal{H}}$. Finally we define the Sobolev space $W_{\mathcal{H}}^{1,p}(H,m)$ by completing 
smooth cylindrical functions w.r.t. the norm
$$
\norm[W_{\mathcal{H}^{1,p}}]{u}:=\norm[L^p]{u}+\norm[L^p]{|D_{\mathcal{H}} u|_{\mathcal{H}}}.
$$
It is not difficult to see that $|D_{\mathcal{H}} u|_{\mathcal{H}}=|Q^{1/2}\nabla u|_H$ for all $u\in W^{1,p}(H,m)$, and thus $W^{1,p}\subset W_{\mathcal{H}}^{1,p}$. In this context Mehler's formula reads
$$
T_t f(x):= \int_H f(e^{-t}x+\sqrt{1-e^{-2t}}y) d m(y)
$$
and $T_t$ is the $L^2(H,m)$-gradient flow associated to the energy 
$\frac{1}{2}\int_H |D_{\mathcal{H}} u|^2_{\mathcal{H}} dm$.  We shall also use the commutation
property 
\begin{equation} \label{eq:commutation_Tt} 
\langle D_{\mathcal{H}} T_tf(x),v\rangle =e^{-t} T_t \langle D_{\mathcal{H}} f,v\rangle(x)
\qquad\text{for $m$-a.e. $x\in H$.} 
\end{equation} 
for all $v\in\mathcal{H}$. A standard reference on this topic is \cite{Ustunel}.

As we did for $P_t$, we still use Mehler's formula to have a pointwise defined version of the semigroup of $T_t$ which
satisfies, as easily seen, the pointwise semigroup property $T_s\circ T_tf(x)=T_{s+t}f(x)$. In addition (see for instance
\cite[Page~237]{Bogaconv}), a monotone class argument shows that if $f$ is Borel and $2$-summable, then 
$t\mapsto T_tf(x)$ is continuous in $(0,\infty)$ and converges to $f(x)$ as $t\to 0$ for $m$-a.e. $x\in H$.
However, one of the main differences with respect to $P_t$ is that, by the lack of absolute continuity
of the shifted measure, $f=g$ $m$-a.e. in $H$ does not imply $T_tf(x)=T_t g(x)$ (while it implies
$T_tf=T_tg$ $m$-a.e. on $X$). We will also need the
following result, analogous to \autoref{thm:verygood}.

\begin{theorem} \label{thm:verygood1} For any $p\in (1,\infty]$, $t>0$ and $f:H\to\setR$ Borel with
$\int_H |f|^p dm<\infty$ the following property
holds: for $m$-a.e. $x_0\in H$ the restriction of the function $T_tf$ to $x_0+\mathcal{H}$ is continuous, finite and 
everywhere Gateaux differentiable. In addition $s\mapsto T_tf(x_0+sh)$ is of class $C^1$ in $\setR$ for all
$h\in\mathcal{H}$.
\end{theorem}

\begin{proof} Let us fix a time parameter $t>0$. We assume without loss of generality that $f(x)\geq 0$ for every $x\in H$ and we set 
	$$
	N:=\left\lbrace x |\ T_t f^p(x)=\infty \right\rbrace,
	$$ 
that is an $m$-negligible set since $T_t f^p\in L^1(H,m)$. 
Let us fix $x\in H\setminus N$, we first prove $T_t f(x+h)$ is finite and continuous at $h$ for every $h\in \mathcal{H}$. 
By Cameron-Martin formula \eqref{eq:CM_Wiener} and H\"older inequality, we have
\begin{align*}
T_tf(x+h)= & \int_H f\left(e^{-t}x+\sqrt{1-e^{-2t}}\left(y+\frac{e^{-t}}{\sqrt{1-e^{-2t}}}h\right)\right) dm(y)\\
=& \int_H f(e^{-t}x+\sqrt{1-e^{-2t}}y)\rho(y,h,t)d m(y)\\
\leq & (T_t f^p(x))^{1/p} \| \rho(\cdot,h ,t) \|_{L^{p'}},
\end{align*}
where
$$
\rho(y,h,t)=\exp\left\lbrace -\frac{1}{2}a_t^2|h|_{\mathcal{H}}^2+a_tW_{Q^{-1/2}h}(y)\right\rbrace,\quad \quad \text{and}\quad\quad  a_t:=\frac{e^{-t}}{\sqrt{1-e^{-2t}}}.
$$
From \eqref{exponential integrability W} we obtain
$$
\| \rho(\cdot,h ,t) \|_{L^{p'}}\leq \left(\int_H e^{p'a_tW_{Q^{-1/2}h}} dm(y)\right)^{1/p'}=e^\frac{a_t |h|_{\mathcal{H}}}{2}<\infty,
$$
and we easily deduce that $T_tf(x+h)<\infty$ and that the map $\mathcal{H}\to L^{p'}(H,m)$, 
$h\to \rho(\cdot,h,t)$ is continuous. This immediately implies the continuity of $T_t f(x+h)$ at every $h\in \mathcal{H}$.

We now prove that $T_tf$ is differentiable, along Cameron-Martin directions, in $x+h$ for $h\in \mathcal{H}$. 
Let us fix $v\in \mathcal{H}$, it is enough to show the differentiability of $s\mapsto T_tf(x+h+sv)$ in $s=0$. 
Starting from the equality
\begin{equation}\label{CM identity}
T_tf(x+h+sv)=\int_H f(e^{-t}x+\sqrt{1-e^{-2t}}y)\rho(y,h+sv,t)d m(y),
\end{equation}
we have only to check that we are allowed to differentiate under the integral, but this is trivial since the map $y\mapsto f(e^{-t}x+\sqrt{1-e^{-2t}}y)$ belongs to $L^p(H,m)$ (i.e. $T_tf^p(x)<\infty$) and the incremental ratios of $s\mapsto \rho(y,h+sv,t)$ are bounded in $L^{p'}(H,m)$.

We finally show that $s\to T_tf(x_0+sh)$ is of class $C^1$. For very $s_0\in \setR$, the differentiability of $s\to T_tf(x_0+sh)$, is guaranteed by the previous step of the proof. Moreover, starting from \eqref{CM identity}, we have the explicit formula
\begin{align*}
\frac{d}{d s} T_tf(x+ & sh)|_{s=s_0}=\\
&\int_H f(e^{-t}x+\sqrt{1-e^{-2t}}y)\rho(y,s_0 h,t)(-a_t^2 |h|^2_{\mathcal{H}}s_0+a_tW_{Q^{-1/2}h}(y)) d m(y);
\end{align*}
it is now simple to check that $s\mapsto \frac{d}{dh}f(x+sh)$ is continuous. 
\qedhere
\end{proof}

Finally, in this context we shall use the validity of the Riesz inequalities \eqref{Riesz1}, \eqref{Riesz2}, see e.g.\ \cite[Chap.~3]{Ustunel} for a proof (using a transference argument from \cite{pisier}). 


\subsection{$RCD(K,\infty)$ spaces}

The third setting we will be dealing with is the one of $\RCD(K,\infty)$ spaces, $K \in \mathbb R$. This class, introduced in \cite{AGS14}
and deeply studied in the last few years, consists of complete and separable metric spaces $(X,d)$ endowed with
a Borel nonnegative measure $m$ satisfying the growth condition $m(B_r(\bar x))\leq ae^{br^2}$ for some 
$\bar x\in X$ and $a,\,b\geq 0$. In these metric measure structures one can build canonically a
convex and $L^2(X,m)$-lower semicontinuous functional $\Ch(f)=\int_X|\nabla f|^2 dm$, called Cheeger energy. 
The Sobolev space $W^{1,2}(X,m)$ is then defined as the finiteness domain of $\Ch$
and the $\RCD(K,\infty)$
property requires that the metric measure space is $\CD(K,\infty)$ according to Lott-Villani and Sturm, and that
$\Ch$ is a quadratic form.

In metric measure spaces there is always a natural ``heat flow'' semigroup $H_t$, namely
 the $L^2(X,m)$ gradient flow of $\tfrac 12 \Ch(f)$, which has a canonical extension to all
 $L^p(X,m)$ spaces, $1\leq p\leq\infty$ \cite{AGSheat}. Now, in $\RCD(K,\infty)$ spaces the quadraticity of $\Ch$ ensures 
that $H_t$ is linear, while the curvature assumption $\CD(K,\infty)$
ensures the identification of $H_t$ with another semigroup: the gradient flow $\mathscr H_t$ 
of the relative entropy in the space $\mathscr P_2(X)$
w.r.t. the Wasserstein distance. More precisely,
the transition probabilities of $H_t$ satisfy
$$
p_t(x,\cdot) m=\mathscr H_t\delta_x\qquad\forall x\in X,\,\,t>0
$$
and, given $t>0$, one can collect versions of $p_t(x,\cdot)$, $x\in X$, in such a way that $p_t$ is 
$m\times m$-measurable.
As a consequence, also in the $\RCD$ setting one has a canonical and pointwise defined version of the semigroup
provided by the densities of $\mathscr H_t\delta_x$, namely $H_tf(x)=\int_X f(y)p_t(x,y) dm(y)$,
and $f=g$ $m$-a.e. imply $H_tf(x)=H_tg(x)$ for all $x\in X$ and all $t>0$. In addition, the semigroup property
of $\mathscr H_t$ yield this particular form of the Chapman-Kolmogorov equation
\begin{equation}\label{eq:Chapman-Kolmogorov}
\int p_s(x,y)p_t(y,z) dm(y)=p_{s+t}(x,z)\quad\text{for $m$-a.e. $z$, for all $x$}
\end{equation}
which implies that this version of the semigroup satisfies $H_s\circ H_tf(x)=H_{s+t}f(x)$.
Since the metric structure is involved in the construction of $\Ch$ and of $H_t$, the infinitesimal
generator of $H_t$ is denoted $\Delta$.

In this setting one
has still the gradient contractivity property $|\nabla  H_tf|\leq e^{-Kt}H_t|\nabla f|$ $m$-a.e. in $X$
\cite{Savare} for $f\in W^{1,2}(X,m)$. We state here the additional regularity properties that are relevant for our proof: 
\begin{itemize}
\item[(a)] for all $f\in L^\infty(X,m)$, $H_tf\in C_b(X)$ for all $t>0$, and $t\mapsto H_tf(x)$ is continuous in $(0,\infty)$;
\item[(b)] $L^\infty-{\rm Lip}$-regularization: for all $t>0$ the semigroup $H_t$ maps $L^\infty(X,m)$ into 
${\rm Lip}_b(X)$ (see \cite{AGS14}, also with the quantitative statement);
\item[(c)] when $f$ is Lipschitz and bounded, the contractivity estimate can be given in the pointwise form
${\rm lip\,}H_t f(x)\leq e^{-Kt} H_t|\nabla f|(x)$, where ${\rm lip}$ is the slope (also called local Lipschitz constant);
\item[(d)]  Wang's infinite-dimensional Harnack inequality (see the $\Gamma$-calculus proof in \cite[Thm.~5.6.1]{bgl},
in the $\RCD(K,\infty)$ setting it can be established along the lines of the proof of Wang's log-Harnack
inequality given in \cite{AGSaop}), for any $g \ge 0$,
\begin{equation}\label{eq:logh}
(H_t g)^\alpha(x)\leq H_t g^\alpha (y)\exp\biggl(\frac{\alpha d^2(x,y)}{2\sigma_K(t)(\alpha-1)}\biggr)\qquad\forall x,\,y\in X
\end{equation}
with $\alpha>1$ and $\sigma_K(t)=K^{-1}(e^{2Kt}-1)$ if $K\neq 0$, $\sigma_0(t)=2t$.
\end{itemize}

We shall also need the extension of (a) from bounded to $2$-integrable functions. In this case one can use monotone
approximation together with the Littlewood-Paley estimate \eqref{eq:LP} to get
\begin{equation}\label{continuityH_t}
\text{$t\mapsto H_tf(x)$ is continuous in $(0,\infty)$ for $m$-a.e. $x\in H$}
\end{equation}
for any Borel and $2$-summable function $f:X\to\setR$.

\section{Lipschitz estimate}

Having in mind \eqref{eq:amb1}, we define
$$
I_t(x_0,x_1):=|R_t f(x_0)-R_t f(x_1)|,\quad \quad J_t(x):=|f(x)-R_tf(x)|,
$$
and we study those functions separately in the three cases of our interest.

\subsection{Estimates of $I_t(x_0,x_1)$}

We start from the case of Da Prato's Sobolev spaces, with $R_t=P_t$.

\begin{proposition}\label{prop: Estimate I for P}
For every $f\in W^{1,p}(H,m)$ with $p\in (1,\infty)$ and $t>0$, one has
\begin{equation}\label{eq:bdpr}
|P_t f(x_0)-P_t f(x_1)|\leq  |x_1-x_0|\exp\left\lbrace \frac{|x_1-x_0|^2}{4t}\right\rbrace(P_t|\nabla f|(x_0)+P_t|\nabla f|(x_1)), 
\end{equation}
for all $x_0,\,x_1\in H$.
\end{proposition}

We set $x_s:=(1-s)x_0+sx_1$, and recalling that $P_tf\in C^{\infty}(H)$ we compute
\begin{align*}
|P_tf(x_1)-P_tf(x_0)|&=\left|\int_0^1 \frac{d}{ds} P_tf(x_s)\ d s\right|\\
&=\int_0^1 \scal{ \nabla P_t f(x_s)}{x_1-x_0} d s\\
&\leq |x_1-x_0|\int_0^1 |\nabla P_t f(x_s)| d s\\
&\leq |x_1-x_0|\int_0^1 P_t|\nabla f|(x_s) d s,
\end{align*}
where we have used the contractivity property \eqref{eq:contra_Pt}.
To conclude the proof of \autoref{prop: Estimate I for P}, we apply the following log-convexity property, so that we can
control the value of $P_t$ at the intermediate points with the value at the endpoints.

\begin{lemma}[Log-convexity of $P_t$]\label{lemma: log convexity of P} 
For every nonnegative Borel function $g:H\to\setR$ and every $t>0$ the map $\log P_t g$ is 
$-\frac{1}{t}$-convex in $H$, i.e.
$$
P_t g((1-s)x_0+sx_1)\leq \exp\left\lbrace \frac{s(1-s)|x_1-x_0|^2}{2t}\right\rbrace(P_tg(x_0))^{1-s}(P_tg(x_1))^s,
$$
for every  $x_0,\,x_1\in H$ and $s\in [0,1]$.
\end{lemma}
\begin{proof}
Setting $x_s:= (1-s) x_0+s x_1$ we can write $P_tg(x_s)=\int_H  g(y)\rho_t(x_s,y) d N_{Q_t}(y)$, where
$\rho_t(x,\cdot)$ is the density of $N_{e^{At}x,Q_t}$ w.r.t. $N_{Q_t}$.
Using the Cameron-Martin formula \eqref{eq:Cameron_Martin} 
together with the linearity of $x\mapsto W_{\Gamma_t x}^t(y)$, we have
\begin{align*}
\rho_t(x_s,y)=&\exp\left\lbrace -\frac{1}{2}|\Gamma_tx_s|^2+\frac{1-s}{2}|\Gamma_tx_0|^2+\frac{s}{2}|\Gamma_tx_1|_H^2 \right\rbrace \rho_t(x_0,y)^{1-s}\rho_t(x_1,y)^s\\
&= \exp\left\lbrace \frac{s(1-s)|\Gamma_t(x_1-x_0)|^2}{2}\right\rbrace \rho_t(x_0,y)^{1-s}\rho_t(x_1,y)^s\\
&\leq \exp\left\lbrace \frac{s(1-s)|x_1-x_0|^2}{2t}\right\rbrace \rho_t(x_0,y)^{1-s}\rho_t(x_1,y)^s,
\end{align*}
where all inequalities are understood for $N_{Q_t}$-a.e. $y$ and we have used the estimate 
$|\Gamma_t|_H\leq 1/\sqrt{t}$.
By H\"older inequality with exponents $q=1/s$, $q'=1/(1-s)$, after integration w.r.t. $N_{Q_t}$ we find 
$$
P_tg(x_s)\leq \exp\left\lbrace \frac{s(1-s)|x_1-x_0|^2}{t}\right\rbrace(P_tg(x_0))^{1-s}(P_tg(x_1))^s.
$$
\end{proof}

With $g=|\nabla f|$, we can now conclude the proof of \autoref{prop: Estimate I for P}:
\begin{align*}
|P_tf(x_1)-P_tf(x_0)|&=|x_1-x_0| \int_0^1 P_t|\nabla f|_H(x_s) ds\\
&\leq |x_1-x_0|\exp\left\lbrace \frac{|x_1-x_0|^2}{4t}\right\rbrace \int_0^1(P_t|\nabla f|(x_0))^{1-s}(P_t|\nabla f|(x_1))^s ds\\
&\leq |x_1-x_0|\exp\left\lbrace \frac{|x_1-x_0|^2}{4t}\right\rbrace(P_t|\nabla f|(x_0)+P_t|\nabla f|(x_1)).
\end{align*}

Notice that a similar estimate of $I_t$ could be given avoiding a direct computation and using, instead, Wang's Harnack 
inequality (\cite{RoWa}). We shall follow this strategy when dealing with $H_t$, in the class of $\RCD(K,\infty)$ spaces
(since in the non-Gaussian setting explicit computations are usually impossible).

Next, we consider the Wiener space case, with $R_t=T_t$. Notice that, unlike \autoref{prop: Estimate I for P}, there is a $m$-negligible 
exceptional set in the inequality, and that it depends on the chosen Borel 
versions of $f$ and $|D_{\mathcal H} f|_{\mathcal H}$, while it is independent
of $h$.

\begin{proposition}\label{prop: Estimate I for T}
For every $p\in (1,\infty)$, $t>0$ and $f\in W_{\mathcal{H}}^{1,p}(H,m)$, one has 
$$
|T_t f(x_0+h)-T_t f(x_0)| \leq  |h|_{\mathcal{H}}e^{\frac{|h|_{\mathcal{H}}^2}{4t}}
\bigl(T_t|D_{\mathcal H} f|_{\mathcal H}(x_0+h)+T_t|D_{\mathcal H} f|_{\mathcal H}(x_0))\quad\forall
h\in\mathcal{H}, 
$$
for $m$-a.e. $x_0 \in H$.
\end{proposition}

The proof works almost exactly as for $T_t$.
Let $x_0\in H$ be a point such that $T_tf$ is 
Gateaux differentiable at $x_0+h$ for all $h\in\mathcal{H}$: by \autoref{thm:verygood1} $m$-a.e. $x_0\in H$ 
has this property. For any such $x_0$ we have
\begin{align*}
|T_tf(x_0+h)-T_tf(x_0)|&=\left|\int_0^1 \frac{d}{ds} T_tf(x+sh)\ d s\right|=\left|\int_0^1 ( D_{\mathcal{H}} T_t f(x+sh),\ h)_{\mathcal{H}} d s\right|\\
&\leq |h|_{\mathcal{H}}\int_0^1 | D_{\mathcal{H}} T_t f(x+sh)|_\mathcal{H} d s.
\end{align*}
We now use the commutation property \eqref{eq:commutation_Tt} of $T_t$, in the scalar contractivity form
$|D_\mathcal{H}T_t f|_\mathcal{H}\leq e^{-t}T_t|D_\mathcal{H}f|_\mathcal{H}$; denoting by $N$ the $m$-negligible set where the inequality
does not hold, a simple application of Fubini's theorem shows that the set
$$
M_h:=\left\{x\in H:\ \Leb^1(\{s\in\setR: x+sh\in N\})>0\right\}
$$
is $m$-negligible, so that
\begin{equation*}
|T_t f(x_0+h)-T_t f(x_0)|\leq |h|_{\mathcal{H}} 
e^{-t}\int_0^1 T_t |D_{\mathcal H} f|_{\mathcal H}(x+sh) ds, \qquad \text{for $m$-a.e. $x_0 \in H$,}
\end{equation*}
thus, applying \autoref{lemma: log convexity of T} below we get
\begin{equation}\label{c_1}
|T_t f(x_0+h)-T_t f(x_0)| \leq  |h|_{\mathcal{H}}e^{\frac{|h|_{\mathcal{H}}^2}{4t}}
\bigl(T_t|D_{\mathcal H} f|_{\mathcal H}(x_0+h)+T_t|D_{\mathcal H} f|_{\mathcal H}(x_0))
\end{equation}
for $m$-a.e. $x_0 \in H$, where a priori the negligible set depends on $h$.

Let now $S$ be a countable dense set in $\mathcal{H}$.
We finally observe that, for every $x_0$ such that \eqref{c_1} holds for all $h\in S$
and the restriction of $T_t|D_{\mathcal H} f|_{\mathcal H}$ to $x_0+\mathcal{H}$ is finite and continuous,
the inequality \eqref{c_1} must be true for every $h\in \mathcal{H}$, since every term in the inequality is continuous with respect to $h$ and $S$ is dense in $\mathcal{H}$. 

\begin{lemma}[Log-convexity of $T_t$]\label{lemma: log convexity of T}
For every nonnegative Borel function $g: H\rightarrow \setR$, for every $t>0$ the map 
$\log T_t g$ is $-\frac{1}{t}$-convex with respect to the Cameron-Martin distance, i.e.
$$
T_t g((1-s)x_0+sx_1)\leq \exp\left\lbrace \frac{s(1-s)}{2t}|x_1-x_0|_\mathcal{H}^2\right\rbrace(T_tg(x_0))^{1-s}(T_tg(x_1))^s,
$$
for every $x_0,\, x_1\in H$ with $x_1-x_0\in \mathcal{H}$ and $s\in [0,1]$.
\end{lemma}

\begin{proof}
We set $x_s:=(1-s)x_0+sx_1$ and $h=(x_1-x_0)$. We denote by $\rho(\cdot,sh,t)$ the density of
$N_{she^{-t}/\sqrt{1-e^{-2t}},Q}$ w.r.t. $m=N_Q$ and estimate
\begin{align*}
T_t g(x_s)=& \int_H g(e^{-t}x_0+\sqrt{1-e^{-2t}}y+she^{-t}) d m(y)\\
=& \int_H g(e^{-t}x_0+\sqrt{1-e^{-2t}}y) \rho(x,sh,t) dm(y)\\
=&\int_H g(e^{-t}x_0+\sqrt{1-e^{-2t}}y)^{1-s}(g(e^{-t}x_0+\sqrt{1-e^{-2t}}y) \rho(x,sh,t)^{1/s} )^s dm(y)\\
\leq & (T_t g(x_0))^{1-s} \left(\int_H g(e^{-t}x_0+\sqrt{1-e^{-2t}}y) \rho(x,sh,t)^{1/s} d m(y)\right)^s,
\end{align*}
where in the last passage we have used the H\"older inequality with exponents $p=1/s$ and $p'=1/(1-s)$.

Now, using the Cameron-Martin formula \eqref{eq:CM_Wiener}, we have
$$
\rho(y,sv,t)=\exp\left\lbrace -\frac{1}{2}\left(\frac{e^{-t}}{\sqrt{1-e^{-2t}}}\right)^2s^2
|h|_{\mathcal{H}}^2+\left(\frac{e^{-t}}{\sqrt{1-e^{-2t}}}\right)sW^t_{Q^{-1/2}h}(y)
\right\rbrace
$$
and a simple computation shows that
$$
\rho(\cdot,sh,t)^{1/s}=\exp\left\lbrace 
\frac{1}{2}\left(\frac{e^{-t}}{\sqrt{1-e^{-2t}}}\right)^2(1-s)|h|_{\mathcal{H}}^2\right\rbrace\frac{dN_{he^{-t}/\sqrt{1-e^{-2t}},Q}}{dm},
$$
so that
\begin{align*}
&\left(\int_H g(e^{-t}x_0+\sqrt{1-e^{-2t}}y) \rho(y,sh,t)^{1/s} d m(y)\right)^s\\
&=\exp\left\lbrace \frac{1}{2}\left(\frac{e^{-t}}{\sqrt{1-e^{-2t}}}\right)^2s(1-s)|h|_{\mathcal{H}}^2\right\rbrace (T_t g(x_1))^s\\
&\leq \exp\left\lbrace \frac{s(1-s)}{2t}|h|_{\mathcal{H}}^2\right\rbrace (T_t g(x_1))^s,
\end{align*}
that implies the stated inequality.
\end{proof}

Finally, we consider the case $R_t=H_t$, i.e. we deal with a $\RCD(K,\infty)$ metric measure space
$(X,d,m)$. In this case we obtain a slightly weaker estimate, compared to the one \eqref{eq:bdpr} available in 
Da Prato's Sobolev spaces, because of the $\alpha$-th power and because it holds for $m$-a.e. $x_0$.

\begin{proposition}\label{prop:ItH}
For every $\alpha\in (1,2]$, $t>0$ and $f\in W^{1,2}(X,m)$ one has
\begin{equation}\label{eq:bdpr1}
|H_t f(x_0)-H_t f(x_1)|\leq  d(x_0,x_1) e^{-Kt}\exp\left\lbrace \frac{d^2(x_0,x_1)}
{2\sigma_K(t)(\alpha-1)} \right\rbrace(H_t|\nabla f|^\alpha(x_0))^{1/\alpha}\qquad\forall x_1\in X. 
\end{equation}
for $m$-a.e. $x_0\in X$.
\end{proposition}
\begin{proof} By a simple truncation argument, it is not restrictive to assume that $f$ is bounded, so that all
functions $g_r=H_rf$, $r>0$ are bounded and Lipschitz. If we establish the pointwise inequality
$$
|H_t g_r(x_0)-H_t g_r(x_1)|\leq  d(x_0,x_1) e^{-Kt}\exp\left\lbrace \frac{d^2(x_0,x_1)}
{2\sigma_K(t)(\alpha-1)} \right\rbrace(H_t|\nabla g_r|^\alpha (x_0))^{1/\alpha}
$$
for all $x_0,\,x_1\in X$ 
we can then pass to the limit as $r\to 0$ and use the pointwise continuity of the semigroup on bounded functions
to achieve \eqref{eq:bdpr1}.

As in the proof of \autoref{prop: Estimate I for P}, with a unit speed geodesics $(x_s)_{s \in [0,1]}$ instead of the linear interpolation, and using the pointwise contractivity estimate
\[{\rm lip\,}H_t g_r(x)\leq e^{-Kt}H_t|\nabla g_r|(x),\] we are led to estimate 
$\int_0^1 H_t|\nabla g_r|(x_s) d s$; using Wang's Harnack inequality 
\eqref{eq:logh} we get
\begin{align*}
\int_0^1 H_t|\nabla g_r|(x_s) d s  \leq & \left(\int_0^1 \bigl(H_t|\nabla g_r|(x_s)\bigr)^\alpha d s\right)^{1/\alpha}\\
\leq & \left(\int_0^1 H_t |\nabla g_r|^\alpha (x_0) \exp\left\lbrace \frac{\alpha d^2(x_0,x_s)}{2\sigma_K(t)(\alpha-1)} \right\rbrace ds\right)^{1/\alpha}\\
\leq & (H_t|\nabla g_r|^\alpha(x_0))^{1/\alpha} \exp\left\lbrace \frac{d^2(x_0,x_1)}{2\sigma_K(t)(\alpha-1)} \right\rbrace. \qedhere
\end{align*}
\end{proof}

\subsection{Estimate of $J_t(x)$}
We look for a pointwise estimate of the form
$$
|R_tf(x)-f(x)|\leq \sqrt{t} g(x),  
$$
where $g$ is a nonnegative function satisfying
$$
\| g\|_{L^p}\leq C_p \|\nabla f\|_{L^p}\quad\quad\quad \forall \ 1<p<\infty.
$$
Natural candidates are $g(x) = \sup_{t>0} R_t|\nabla f|(x)$, as in the finite-dimensional theory, 
or $g(x) =  \sup_{t>0}  |R_t \sqrt{-L}f|(x)$. Here we focus on the latter, starting from
\autoref{prop: fractional estimate Markov semigroups}, and considering the three cases
of our interest.

\begin{theorem}[Estimate of $J_t(x)$, Da Prato's case]\label{th: Estimate II for P}
Let $p \in (1, \infty)$. For every $f\in W^{1,p}(H,m)$ the function
$$
\tilde{f}(x):=
\begin{cases}
\lim\limits_{t\rightarrow  0} P_t f(x) & \text{when it exists,}\\
0 & \text{elsewhere}
\end{cases}
$$
provides a canonical representative of $f$ such that, for all $t \ge 0$,
$$
|e^{-t} P_tf(x)-\tilde f(x)|\leq  \frac{4\sqrt{t}}{\sqrt{\pi}} \sup_{s>0} |P_s\sqrt{I-L}f|(x) \quad\quad \forall x\in H,
$$
where $L$ is the infinitesimal generator of $P_t$.
\end{theorem}

\begin{remark}
By Riesz inequality \eqref{Riesz2}, one has  $\sqrt{I-L}f \in L^{p}(H, m)$ for $f \in W^{1,p}(H,m)$. The maximal inequality (\autoref{thm:maximal}) implies that $\sup_{s>0} P_s |\sqrt{I-L}f| \in L^p(H, m)$, hence we obtain the inequality
\begin{equation}\label{eq:quantitative-bound-sup-riesz} \left\| \sup_{s>0} |P_s\sqrt{I-L}f| \right\|_{L^p} \le c_p \| f\|_{W^{1,p}}.\end{equation}
\end{remark}

\begin{proof}


We apply \eqref{fractional identity on functions} with the Markov semigroup $R_t := e^{-t}P_t$, with generator $I-L$. Inequalities \eqref{Riesz1} and \eqref{Riesz2} and the density of smooth cylindrical functions in $W^{1,p}(H, m)$ imply that, with the notation in  
\autoref{prop: fractional estimate Markov semigroups}, $W^{1,p}(H,m) = D_p(\sqrt{I-L})$. Hence, given $g \in W^{1,p}(H,m)$,
$$
e^{-t} P_tg(x)-g(x)=\int_0^{\infty} K(s,t) e^{-s} P_s\sqrt{I-L}g (x) d s, \quad\quad \text{for $m$-a.e. $x\in H$,}
$$
and we observe that the equality holds pointwise in $H$ if $g=e^{-h}P_h f$, since then both sides are 
continuous (as a simple application of dominated convergence for the right hand side) and $m$ has full support.
Therefore for every $x\in H$ we have
\begin{align*}
|e^{-(t+h)}P_{t+h}f(x)-e^{-h}P_hf(x)|=&\left|\int_0^{\infty}K(s,t) e^{-s} P_s\sqrt{I-L} e^{-h} P_hf (x) d s\right|\\
=&\left|\int_0^{\infty} K(s,t) e^{-s+h} P_{s+h}\sqrt{I-L}f (x) d s\right|\\
\leq & \int_0^{\infty}|K(s,t)| d s\ \sup_{s>0} |P_s\sqrt{I-L}f|(x)\\
=&\frac{4\sqrt{t}}{\sqrt{\pi}} \sup_{s>0} |P_s\sqrt{I-L}f|(x).
\end{align*}
This implies the existence of $\lim_{t\rightarrow 0}P_tf(x)$ when $\sup_{t>0} |P_t \sqrt{I-L}f|(x)<\infty$.
\end{proof}

In the following remark we illustrate how the proof and the statement 
of \autoref{th: Estimate II for P} need to adapted to the cases of $T_t$ and $H_t$, semigroups 
which have weaker regularizing properties.

\begin{remark}[Estimate of $J_t(x)$, Wiener space case]\label{remarkJt-T}
In the case of $T_t$, for $f\in W^{1,p}_{\mathcal{H}}(H,m)$ ($p \in (1,\infty)$) we get with a similar argument (using Riesz inequalities \eqref{Riesz1} and \eqref{Riesz2} in this setting) 
\begin{equation}\label{eq:JestimateHt}
|e^{-t} T_tf(x)-f(x)|\leq  \frac{4\sqrt{t}}{\sqrt{\pi}} \sup_{s>0} |T_s\sqrt{I-L}f|(x) 
\qquad\text{for $m$-a.e. $x\in H$,}
\end{equation}
where $L$ is the infinitesimal generator of $T_t$, provided we choose Borel representatives of $f$ and $\sqrt{I-L}f$.
\end{remark}

\begin{remark}[Estimate of $J_t(x)$, $\RCD(K, \infty)$ case]\label{remarkJt-H}
In this case, we limit ourselves to the case $p=2$, since, to the authors' knowledge, Riesz inequalities in this setting are not known, although strongly expected to hold (see e.g.\ the seminal work \cite{bakry}, and also \cite{jiang} for the case of finite dimension). For every $\varphi\in D(\Delta)$ (that is a dense subset of $W^{1,2}$) we have
\begin{equation}\label{eq:ibp-riesz-2}
\norm[L^2]{\sqrt{-\Delta}\varphi}^2=\int_X \sqrt{-\Delta}\varphi \sqrt{-\Delta}\varphi d m=-\int_X \varphi \Delta \varphi d m
= \int_X |\nabla\varphi|^2 d m,
\end{equation}
so that $W^{1,2}(X, m) \subset D(\sqrt{-\Delta})$. Hence, replicating the argument of \autoref{th: Estimate II for P}, we get
\begin{equation}\label{eq:JestimateHtbis}
|H_tf(x)-f(x)|\leq  \frac{4\sqrt{t}}{\sqrt{\pi}} \sup_{s>0} |H_s\sqrt{-\Delta}f|(x) 
\qquad\text{for $m$-a.e. $x\in X$.}
\end{equation}
\end{remark}

\section{Lipschitz approximation, in Lusin's sense, of Sobolev functions}

We prove our main Lipschitz approximation result, considering first the case of the Sobolev spaces $W^{1,p}(H,m)$, then the spaces
$W^{1,p}_{\mathcal{H}}(H,m)$ and finally the spaces $W^{1,2}(X,m)$. In the statement, even though this would not be necessary for the
cases (1) and (3), it is understood that the semigroups appearing in the definition of $g$ 
act on Borel representatives, so that the estimates of $J_t(x)$ given in the previous section are applicable.

\begin{theorem}\label{th: Lipschitz approximation scalar}Let $p\in (1,\infty)$.
\begin{itemize}
\item[(1)] For every $f\in W^{1,p}(H,m)$ one has
\begin{equation}\label{eq:12dec1}
|f(x)-f(y)|\leq C_p |x-y| (g(x)+g(y))\quad\quad\quad \forall x,\,y\in H,
\end{equation}
with $g=\sup_{t>0}P_t|\nabla f|+\sup_{t>0}|P_t\sqrt{I-L} f| \in L^p(H, m)$, $L$ being the infinitesimal generator of $P_t$.
\item[(2)] For every $f\in W^{1,p}_{\mathcal{H}}(H,m)$, for $m$-a.e. $x\in H$ one has
\begin{equation}\label{eq:12dec2}
|f(x+h)-f(x)|\leq C_p |h|_{\mathcal{H}}(g(x+h)+g(x))\qquad\forall h\in \mathcal{H},
\end{equation}
with $g=\sup_{t>0}T_t|D_{\mathcal{H}} f|_{\mathcal{H}}+\sup_{t>0}|T_t\sqrt{I-L} f| \in L^p(H, m)$, $L$ being the infinitesimal generator of $T_t$.
\item[(3)] For every $\alpha \in (1,\ 2)$, $f\in W^{1,2}(X,m)$, there exists a $m$-negligible set $N\subset X$ such that, for every $x,y\in X\setminus N$ with $d(x,y)\leq 1/(K^-)^2$ one has
\begin{equation}\label{eq:12dec3}
|f(x)-f(y)|\leq C_\alpha d(x,y) (g(x)+g(y)),
\end{equation}
with $g=\left( \sup_{t>0}H_t|\nabla f|^{\alpha}\right)^{1/\alpha}+\sup_{t>0}|H_t\sqrt{-\Delta} f| \in L^2(X,m)$, where $K^{-}$ denotes the negative part of the curvature bound $K$.
\end{itemize}
\end{theorem}
\begin{proof} To prove (1), it suffices to use the decomposition \eqref{eq:amb1} with $R_t := e^{-t} P_t$ and apply 
\autoref{prop: Estimate I for P} and \autoref{th: Estimate II for P} with $t=|x-y|^2$. Similarly, to prove  (2), we let $R_t := e^{-t} T_t$ and use \autoref{prop: Estimate I for T} and \autoref{remarkJt-T}. Finally, the proof of (3) requires little work.
Using the decomposition \eqref{eq:amb1} with $R_t := H_t$, applying \autoref{prop:ItH} and \autoref{remarkJt-H}, with $t=d^2(x,y)$, we find an $m$-negligible set $N$ such that
\begin{align*}
|f(x)-f(y)| &\leq d(x,y) \frac{4}{\pi}(\sup_{t>0}|H_t\sqrt{-\Delta} f|(x)+\sup_{t>0}|H_t\sqrt{-\Delta} f|(y))\\
+&d(x,y)\exp\left\lbrace -Kd^2(x,y)+\frac{d^2(x,y)}{2\sigma_K(d^2(x,y))(\alpha-1)} \right\rbrace \sup_{t>0}H_t|\nabla f|^{\alpha}(x),
\end{align*}
for every $x,y\in X\setminus N$; where $\sigma_K(t)=K^{-1}(e^{2Kt}-1)$ if $K\neq 0$, $\sigma_0(t)=2t$. In order to conclude the proof we show that
\begin{equation*}
\exp\left\lbrace -Kd^2(x,y)+\frac{d^2(x,y)}{2\sigma_K(d^2(x,y))(\alpha-1)} \right\rbrace\leq C_{\alpha},	
\end{equation*}
for every $x,y\in X$ satisfying $d(x,y)\leq 1/(K^-)^2$.
When $K\geq 0$, using that $\sigma_K(t)\geq 2t$, we obtain
$$
-Kd^2(x,y)+\frac{d^2(x,y)}{2\sigma_K(d^2(x,y))(\alpha-1)}\leq \frac{1}{4(1-\alpha)},
$$
for every $x,y\in X$.
When $K<0$ it is elementary to see that $\sigma_K(t)\geq 2t e^{2Kt}$, thus we deduce
\begin{equation}
-Kd^2(x,y)+\frac{d^2(x,y)}{2\sigma_K(d^2(x,y))(\alpha-1)}\leq -Kd^2(x,y)+\frac{1}{4(1-\alpha)}e^{-2K d^2(x,y)},
\end{equation}
for every $x,y\in X$. The thesis easily follows using that $d(x,y)\leq 1/(K^-)^2$.
\end{proof}

We conclude this section by noticing that vector-valued versions of the previous results hold. 
Given a separable Hilbert space $E$ endowed with a norm $|\cdot |_E$, we define the class of Sobolev $E$-valued maps as follows.

\begin{definition}[$E$-valued Sobolev maps]
Let $p\in [1,\infty)$ and $f\in L^p(H,m;E)$. 
We say that $f\in W^{1,p}(H,m;E)$ if for every $v\in E$ the function $(f,\ v)_E$ belongs to $W^{1,p}(H,m)$ and 
$$
|\nabla f|(x):=\sqrt{\sum_i |\nabla f_i(x)|^2}\in L^p(H,m), 
$$
where we have denoted by $f_i$ the components of $f$ with respect to an Hilbert basis of $E$.
We observe that $|\nabla f|$ does not depend on the choice of the basis (in fact, it is the Hilbert-Schmidt norm on $E \otimes H$). 

The definitions of the Sobolev spaces $W^{1,p}_{\mathcal{H}}(H,m; E)$ and $W^{1,2}(X,m;E)$ are completely analogous.
\end{definition}

In the case $p=2$ it is immediately seen, arguing componentwise, that \eqref{eq:12dec1} holds for
$f\in W^{1,2}(H,m)$, with $g=\sup_{t>0}P_t|\nabla f|+\sup_{t>0}|P_t\sqrt{I-L} f|$ (understanding the action of
$P_t$ and $\sqrt{I-L}$ componentwise). The same holds for \eqref{eq:12dec2} and \eqref{eq:12dec3}.
In the case when $p\neq 2$, with $1<p<\infty$, the argument requires the validity of the Riesz inequalities also for $E$-valued maps.
It is a well-known principle in harmonic analysis that inequalities for singular integrals, such as \eqref{Riesz1} and \eqref{Riesz2}, hold also for maps with values in Hilbert spaces $E$ \cite[\S II.5, Thm.~5]{Stein2}, \cite{burkholder}, where the singular integral operator is applied componentwise. In the case of Riesz inequalities, this can be seen by careful inspection of the proofs provided in the references, both in Da Prato's and in the Wiener space setting. Therefore, even in the $E$-valued setting, 
\eqref{eq:12dec1} and \eqref{eq:12dec2} hold also for general powers $p\in (1,\infty)$.

\section{Quantitative estimates for regular Lagrangian flows}

In this section we provide an application of the Lipschitz approximation result, along the lines of Crippa-DeLellis'
quantitative estimates \cite{crippade}. 
We start recalling the notion of regular Lagrangian flow associated to a Borel vector field $b:[0,T]\times H\rightarrow H$,
satisfying the integrability condition
\begin{equation}\label{eq:18nov}
\int_0^T\int_H|b(t,x)| dtdm(x)<\infty.
\end{equation}

\begin{definition}\label{def:reg-flow}
We say that a map $X:[0,T]\times H \rightarrow H$ is a regular Lagrangian flow associated to $b$ if the following two conditions are satisfied:
\begin{itemize}
\item[(1)] for $m$-a.e.\ $x\in H$ the curve $t\mapsto X(t,x)$ is absolutely continuous and satisfies
\[
X(t,x)=x+\int_0^tb(s,X(s,x)) d s\qquad\text{for all $t\in [0,T]$;}
\]
\item[(2)] there exists a constant $L$, called \textit{compressibility constant}, such that
\[
{X(t,\cdot)}_{\#}m\leq L m \quad \quad \forall  t\in [0,T].
\]
\end{itemize}
\end{definition}

Condition (2) ensures the integrability of $|b(\cdot,X(\cdot,x))|$ in $(0,T)$ for $m$-a.e.\ $x$, since 
$$
\int_H \int_0^T|b(s,X(s,x))| ds dm(x)=
\int_0^T\int_H |b(s,X(s,x))| dm(x) ds\leq L\int_0^T\int_H|b| dm dt<\infty.
$$
A similar argument based on Fubini's theorem shows that, whenever $b=\bar b$ $\Leb^1\times m$-a.e. in $(0,T)\times H$,
$X$ is a regular Lagrangian flow relative to $b$ if and only if it is a relative Lagrangian flow relative to $\bar b$. 

In the sequel we also use the abbreviation $X_t$ for $X(t,\cdot)$.

\begin{theorem}\label{th: Stability} Let $p\in (1,\infty]$ and let $X$ and $\bar{X}$ be regular Lagrangian flows with compressibility 
constants respectively $L$ and $\bar{L}$
associated to the vector fields $$b\in L^1((0,T), W^{1,p}(H,m;H)),\qquad \bar{b}\in L^1((0,T), L^1(H,m;H)),$$ 
with $\|b-\bar b\|_{L^1((0,T)\times H)}<1$. Then
\begin{equation}\label{eq:bellastima}
\int_H |X_t-\bar{X}_t|\wedge 1 d m\leq \frac{C}{| \log \bigl(\|b-\bar{b}\|_{L^1((0,T)\times H;H)}\bigr)\bigr|}\quad \quad \quad \forall t \in [0,T],
\end{equation}
where $C=C_p(L+\bar{L})\|b\|_{L^1(W^{1,p})}+2L+1$.
\end{theorem}


\begin{proof} Following \cite{crippade} we define 
$$
\Phi(t):=\int_H \log \left( \frac{|X_t-\bar{X}_t|}{\delta}+1\right) d m,
$$
and differentiate, to get
\begin{align*}
\Phi '(t) \leq & \int_H \frac{|b_t(X_t)-\bar{b}_t(\bar{X}_t)|}{\delta+|X_t-\bar{X_t}|} d m\\
\leq & \int_H \frac{|b_t(X_t)-b_t(\bar{X}_t)|}{|X_t-\bar{X}_t|} d m + \frac{1}{\delta}\int_H |b_t(\bar{X}_t)-\bar{b}_t(\bar{X}_t)| d m.
\end{align*}
Using Theorem \ref{th: Lipschitz approximation scalar} (in the $H$-valued case) the first term can be estimated, with $g := \sup_{s>0}P_s|\nabla f|+\sup_{s>0}|P_s\sqrt{I-L} f|$, as
\begin{equation}\label{eq:crippa-delellis-first-term} \begin{split}
\int_H \frac{|b_t(X_t)-b_t(\bar{X}_t)|}{|X_t-\bar{X}_t|} d m & \leq \int_H \frac{C_p |X_t - \bar{X_t}| \left( g(X_t) + g(\bar X _t) \right)} {|X_t - \bar{X}_t|} d m \\
& \le C_p \int_{H}\left( g(X_t) + g(\bar X_t)  \right) d m\\
& \le C_p(L+\bar{L})\|b_t\|_{W^{1,p}}.
\end{split}\end{equation}
For the second term, choosing $\delta=\|b-\bar{b}\|_{L^1((0,T)\times H;H)}$, we have
\begin{equation}\label{eq:crippa-delellis-second-term}
 \frac{1}{\delta}\int_H |b_t(\bar{X}_t)-\bar{b}_t(\bar{X}_t)| d m\leq \frac{\bar{L}}{\|b-\bar{b}\|_{L^1((0,T)\times H;H)}}\int_H| b_t-\bar{b}_t| d m.
\end{equation}
Since $\Phi(0)=0$, by integration we get
\begin{align*}
\Phi(t)\leq  &C_p(L+\bar{L})\int_0^t \|b_s\|_{W^{1,p}} ds+\frac{\bar{L}}{\|b-\bar{b}\|_{L^1((0,T)\times H;H)}}\int_0^t\int_H| b_s-\bar{b}_s| d m ds\\
=  & C_p(L+\bar{L}) \|b\|_{L^1(W^{1,p})}+\bar{L}= :C_1.
\end{align*}
We now exploit the uniform bound on $\Phi(t)$ to estimate $\int_H |X_t-\bar{X}_t|\wedge 1 d m$ from above.
Given $s>0$, set 
$$
E_s:=\left\lbrace x \, : \, \log \left( \frac{|X_t-\bar{X}_t|}{\delta}+1\right) >s \right\rbrace.
$$
By Chebychev inequality we have $m(E_s)\leq C_1/s$ and, for every $x\in H\setminus E_s$, our choice of
$\delta$ gives
$$
|X_t(x)-\bar{X}_t(x)|\leq e^s\|b-\bar b\|_{L^1((0,T)\times H;H)}.
$$
We estimate
\begin{align}\label{estimate 1}
\int_H |X_t-\bar{X}_t|\wedge 1 d m= &\int_{E_s} |X_t-\bar{X}_t|\wedge 1 d m+\int_{X\setminus E_s}|X_t-\bar{X}_t| d m\\
\leq & \frac{C_1}{s}+ e^s\|b-\bar b\|_{L^1((0,T)\times H;H)}.
\end{align}

Under the assumption that $\|b-\bar b\|_{L^1((0,T)\times H;H)}<1$, we set 
$$
s:=-\frac{\log(\|b-\bar b\|_{L^1((0,T)\times H;H)})}{2},
$$ and we obtain (using the inequality $\sqrt{z}\geq\log z$ for $z\geq 1$)
$$
e^s\|b-\bar b\|_{L^1((0,T)\times H;H)}=\|b-\bar{b}\|^{1/2}_{L^1((0,T)\times H;H)}\leq |\log(\|b-\bar{b}\|_{L^1((0,T)\times H;H)})|^{-1},
$$
and thus \eqref{eq:bellastima} with $C=2C_1+1$.
\end{proof}

%
%

Plenty of variants of the fundamental argument above can be devised, leading to results under different assumptions
on growth/integrability of the vector fields, of their derivative, and of their divergence. Below, we informally 
discuss some of these ones,  noticing that they can be also combined together.

\begin{remark}[$L^r$-regular Lagrangian flows]\label{rem:r-regular-flows}
In \cite{ambrosio-fig} the wider concept of $L^r$-regular Lagrangian flow is considered (with $r \in [1, \infty]$) replacing (2) in \autoref{def:reg-flow} with the requirement that, for all $t \in [0,T]$, one has ${X(t,\cdot)}_{\#}m = u_t m$  with $L:=\sup_{t \in [0,T]} \| u_t \|_{L^r} < \infty$. For $r \in [1, \infty)$ we can prove the analogue of \autoref{th: Stability} for $L^r$-regular Lagrangian flows, assuming that, for some $p \ge r' := r/(r-1)$,
\[ b \in L^1( (0,T), W^{1,p}(H, m; H)) \quad \tilde{b}\in L^1( (0,T), L^p(H, m; H) ),\]
the thesis being \eqref{eq:bellastima} with $\| b - \tilde{b}\|_{L^1( (0,T), L^p(H, m; H) )}$ in place of of $\| b - \tilde{b}\|_{L^1( (0,T)\times H; H) )}$. The proof goes along the same lines,  applying H\"older inequality (with exponents $r$ and $p$) to obtain the analogues of inequalities \eqref{eq:crippa-delellis-first-term} and \eqref{eq:crippa-delellis-second-term}.
 
Another possible choice is to let instead of (2), ${X(t,\cdot)}_{\#}m = u_t m$ with $L:= \| u_t \|_{L^r( (0,T) \times H)} < \infty$. An adaptation of the proof of \autoref{th: Stability} yields quantitative uniqueness, provided that $b \in L^p( (0,T), W^{1,p}(H, m; H))$ and $\tilde{b}\in L^p( (0,T) \times H )$  with $p \ge r'$.
\end{remark}

\begin{remark}[Generalized flows and continuity equations]\label{rem:continuity-equation}
The proof of \autoref{th: Stability} still holds if we relax the notion of regular Lagrangian flow, replacing the map $X$ with a probability measure $\eta$ on $C([0,T], H)$, concentrated on (absolutely) continuous paths which solve the differential equation associated to $b$ and such that the $1$-marginals $\eta_t := (e_t)_{\#} \eta \le L m$ for all $t \in [0,T]$ ($e_t$ being the evaluation map at $t$, $e_t :C([0,T], H) \to  H$).  Such \emph{generalized flows} often appear as an intermediate step when arguing well-posedness of regular Lagrangian flows, see e.g.\ \cite[Definition 4.1]{ambrosio-fig}. It is also known that the $1$-marginals $(\eta_t)_{ t \in [0,T]}$ of any generalized flow  solve the continuity equation $\partial_t \eta + \operatorname{div} (b \eta) = 0$ in a suitable weak sense. The converse is also true, by the so-called superposition principle (\cite{ambrosio-fig} in the Wiener setting,  and \cite{ambrosio-trevisan, trevisan-stepanov} for more general settings). Thanks to this correspondence, the argument in \autoref{th: Stability} can be used to prove (quantitative) uniqueness of (probability-valued) solutions to the continuity equation of a Sobolev vector field. 
\end{remark}

\begin{remark} [$\mathcal{H}$-valued vector fields]\label{cm-valued-fields}
If we consider the case of $\mathcal{H}$-valued vector fields, first of all we need to remark that the definition
of regular Lagrangian flow and its consequence \eqref{eq:18nov} grant, in this circumstance, a stronger property, 
namely the absolute continuity of $t\mapsto X_t(x)$ w.r.t. the stronger norm $|\cdot|_{\mathcal H}$ for $m$-a.e. 
$x\in H$. In particular, since $X_0(x)-\bar{X}_0(x)=x-x=0$, this gives that $X_t(x)-\bar{X}_t(x)\in \mathcal{H}$ 
in $[0,T]$, for $m$-a.e. $x\in X$.

Starting from this
observation we may replicate the argument of the previous theorem: by differentiation of the functional
$$
\tilde{\Phi}(t):=\int_H \log \left( \frac{|X_t-\bar{X}_t|_{\mathcal{H}}}{\delta}+1\right) d m
$$
we can prove the following result:
if
\[ b \in L^1((0,T), W_{\mathcal{H}}^{1,p}(H,m;\mathcal{H})), \quad \bar{b}\in L^1((0,T), L^1(H,m;\mathcal{H})),\]
and $\|b-\bar b\|_{L^1((0,T)\times H;\mathcal H)}<1$, then
$$
\int_H |X_t-\bar{X}_t|_{\mathcal{H}}\wedge 1 d m\leq \frac{C}{|\log \bigl(\|b-\bar{b}\|_{L^1((0,T)\times H;\mathcal H)}\bigr)\bigr|}
\quad \quad \quad \forall t\in [0,T],
$$
where $C= C_p(L+\bar{L})\|b\|_{L^1(W^{1,p}_{\mathcal H})}+2L+1$.
\end{remark}

\end{document}